\documentclass[12pt]{amsart}
\usepackage{amsmath,amsfonts,amsthm,amssymb}
\usepackage{graphicx}

\newcommand{\C}{\mathbb{C}}

\newcommand{\F}{\mathcal{F}}

\newcommand{\D}{\mathcal{D}}
\newcommand{\dir}{\mathcal{E}}

\newcommand{\ie}{{\em i.$\,$e.$\!$} }

\newcommand{\qand}{\quad\text{and}\quad}
\newcommand{\qqand}{\qquad\text{and}\qquad}

\newtheorem{theorem}{Theorem}[section]

\newtheorem{proposition}[theorem]{Proposition}
\newtheorem{algorithm}[theorem]{Algorithm}

\theoremstyle{definition}

\theoremstyle{remark}

\numberwithin{equation}{section}

\makeindex

\begin{document}

\title[Restrictions of functions on the Hata set]{Restrictions 
of harmonic functions and Dirichlet eigenfunctions of the Hata set to the 
interval}

\author{Baltazar Espinoza}
\address{Facultad de Ciencias\\Universidad de Colima\\
Colima, Colima, Mexico\\28045}
\email{espfable@gmail.com}
\author{Ricardo A. S\'aenz}
\address{Facultad de Ciencias\\Universidad de Colima\\
Colima, Colima, Mexico\\28045}
\email{rasaenz@ucol.mx}
\date{\today}

\begin{abstract}
In this paper we study the harmonic functions and the Dirichlet 
eigenfunctions of the Hata set, and their restrictions to the interval $[0,1]$, 
its main edge. We prove that these restrictions of the harmonic functions are 
singular, \ie monotone and with zero derivatives almost everywhere, and provide 
numerical evidence that this is also the case for the eigenfunctions.
\end{abstract}

\maketitle


\section{Introduction}

Interest in the study of analysis in fractals has increased since the 
publication of Kigami's papers \cite{Kigami89} and \cite{Kigami93}. In 
particular, there has been interest in the explicit construction of harmonic 
functions and the eigenfunctions of the Laplacian of a postcritically 
finite (PCF) set. In \cite{DSV99}, Dalrymple, Strichartz and Vinson described
algorithms for the construction of harmonic functions and the 
eigenfunctions in the Sierpinski gasket. The construction of harmonic functions 
is achieved by the general algorithm for PCF sets described in \cite[Chapter 
3]{Kigami}, while the 
construction of the eigenfunctions is achieved by decimation \cite{Shima91} 
(see also \cite[Chapter 3]{StrichartzDEF} for a detailed explanation of the 
decimation method). The explicit construction of harmonic functions or 
eigenfunctions has been done in other fractals, as the Vicset set \cite{CSW11}, 
where one also has decimation, and the pentakun \cite{ASST03}, where one does 
not have decimation.

In \cite{DSV99}, the authors also described algorithms for the restriction of 
harmonic functions and the eigenfunctions to the edges of the Sierpinski
gasket, allowing us to visualize them as functions on the interval $[0,1]$. 
Demir, Dzhafarov, Ko{\c{c}}ak and {\"U}reyen \cite{DDKU07} observed that these
functions have zero derivatives on a dense subset of $[0,1]$. Later De Amo, 
D\'iaz Carrillo and Fern\'andez S\'anchez \cite{ADF13} proved that such 
restrictions are singular functions whenever they are monotone, \ie that have 
zero derivatives almost everywhere.

In this paper we construct the harmonic functions and the Dirichlet 
eigenfunctions for the harmonic structure of the Hata set \cite{Hata85}, and 
their restrictions to the interval $[0,1]$, the longest edge contained in the 
set. Since the Hata set does not have decimation, we have to construct the 
eigenfunctions by the finite element method \cite{DSV99}. Moreover, the Hata 
set has a natural family of harmonic structures, so we study the properties of 
these restrictions with respect to the parameter of the family.

In section \ref{structuresection}, we describe the family of harmonic structures 
of the Hata set, as well as explicitly calculate its Laplacians. In section 
\ref{harmonicsection} we construct the harmonic functions, its restrictions, and 
study whether these restrictions are singular functions. In section 
\ref{Laplaciansection} we explicitly describe the Laplacian with respect to a 
self-similar homogenous measure, and in section \ref{Dirichletsection} we study 
the Dirichlet eigenfunctions, as well as its restrictions, and give numerical 
evidence to decide whether they are also singular.

\section{Harmonic structure}\label{structuresection}

The Hata tree set is the unique compact set $K\subset\C$ such that
\[
K = F_1(K) \cup F_2(K),
\]
where the functions $F_1, F_2$ are given by
\[
F_1(z) = \alpha \bar z \qqand F_2(z) = (1 - |\alpha|^2)\bar z + |\alpha|^2,
\]
and $\alpha\in\C$ is such that $0 < |\alpha|,|1-\alpha| < 1$
\cite{Hata85,YHK97}. Observe that the points $0$ and $1$ are the fixed points 
of $F_1$ and $F_2$, respectively, $\alpha = F_1(1)$, and 
\[
|\alpha|^2 = F_1(\alpha) = F_2(0).
\]
Hence, the critical set is given by $\mathcal C = \{|\alpha|^2\}$ and the
post-critical set, its \emph{boundary}, is
\[
V_0 = \{\alpha, 0, 1\}.
\]
(See Figure \ref{Hata-fig}.) 
\begin{figure}[ht]
\includegraphics[width=3in]{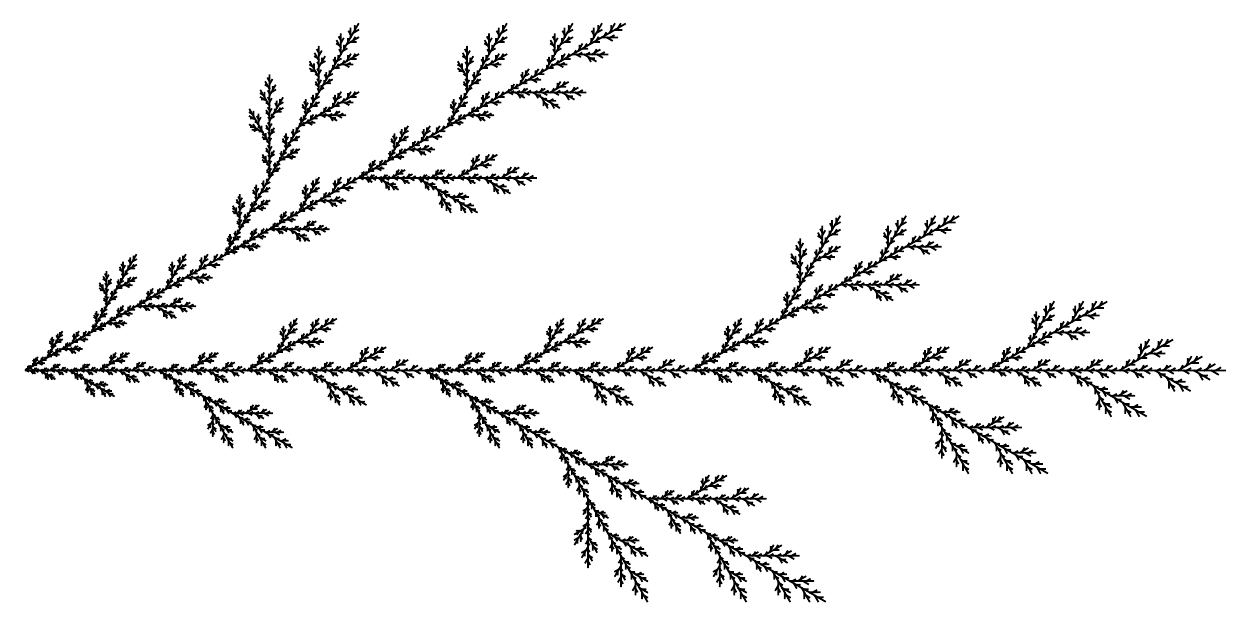}
\caption{Hata tree set, with $\alpha=1/2 + \sqrt{3}i/6$. The critical set
consists of the point $|\alpha|^2 = 1/3$.}
\label{Hata-fig}
\end{figure}
We will denote the points $\alpha, 0$ and $1$ by $p_0, p_1$ and $p_2$, 
respectively.

We observe that $K\setminus V_0$ is disconnected. We call $L$ the closure of the 
connected component of $K\setminus V_0$ that 
contains the interval $(0,1)$, and $M$ the closure of the component that 
contains the open segment from $0$ to $\alpha$.

If $w\in W_m$, we define
$p_{wi} = F_w(p_i)$, for $i=0,1,2$. We note that $p_0 = p_{12}$, and that
\[
|\alpha|^2 = p_{10} = p_{21}.
\]
As points in $V_m$, $m\ge1$, we see that $p_0 = p_{12\ldots2}$,
$p_1=p_{1\ldots1}$ and $p_2 = p_{2\ldots2}$.

If $w\not=w'$ are in $W_m$ and $p\in K_w\cap K_{w'}$, then either
\[
p = p_{w1} = p_{w'0} \qquad\text{or}\qquad p = p_{w1} = p_{w'2}.
\]
We have the former case if $p\in V_m\setminus V_{m-1}$, and the latter if $p\in
V_{m-1}$. If $p\in K_w$, and in no other $K_{w'}$, then
\[
p = p_{w0} \qquad\text{or}\qquad p = p_{w2}.
\]
Again, we have the former case if $p\in V_m\setminus V_{m-1}$, and the latter if
$p\in V_{m-1}$ (and $p\not= p_1$).

A point in $V_m$ has only one adyacent vertex if it's of the form $p_{w20}$ or 
$p_{w22}$; otherwise it has three adyacent vertices in $V_m$, except for $p_1$, 
which has only two.\footnote{We note that, since $p_0 = p_{12}$, 
$p_{w20} = p_{w212}$ and $p_{w10} = p_{w112}$. Thus, points of the form 
$p_{w12}$ may have one or three adyacent vertices, depending on the last term 
of $w$.}

To construct a harmonic structure on the Hata set $K$, we need a Laplacian $D$
on $V_0$. Using the standard base $\{\chi_\alpha, \chi_0, \chi_1\}$, we set
\[
D =
\begin{pmatrix}
-h & h & 0\\
h & -(h+1) & 1\\
0 & 1 & -1
\end{pmatrix}.
\]
Then, if $\mathbf r = (r_1, r_2)$, $(D,\mathbf r)$ is a regular harmonic 
structure \cite{Kigami} for $K$ if $h > 1$, 
\[
r_1 = \frac{1}{h} \qqand r_2 = 1 - \frac{1}{h^2}.
\]
Explicitly, the Laplacian $H_m$ at a point $p\in V_m\setminus V_0$
is given by, if $p=p_{w1}=p_{w'0}$,
\begin{subequations}\label{LaplacianPoint}
\begin{equation}\label{LaplacianPointNew}
\begin{split}
H_mf(p) =
\frac{1}{r_w} & \Big(f(p_{w2}) - f(p_{w1}) + h\big(f(p_{w0}) -
f(p_{w1})\big)\Big)
\\
&+ \frac{h}{r_{w'}}\Big(f(p_{w'1}) - f(p_{w'0})\Big);
\end{split}
\end{equation}
if $p = p_{w1} = p_{w'2}$,
\begin{equation}\label{LaplacianPointOld}
\begin{split}
H_mf(p) =
\frac{1}{r_w} & \Big(f(p_{w2}) - f(p_{w1}) + h\big(f(p_{w0}) -
f(p_{w1})\big)\Big)
\\
&+ \frac{1}{r_{w'}}\Big(f(p_{w'1}) - f(p_{w'2})\Big);
\end{split}
\end{equation}
if $p=p_{w0}$ and in no other cell,
\begin{equation}\label{LaplacianPointNewOne}
H_mf(p) = \frac{h}{r_{w}}\Big(f(p_{w1}) - f(p_{w0})\Big);
\end{equation}
and, if $p=p_{w2}$ and in no other cell,
\begin{equation}\label{LaplacianPointOldOne}
H_mf(p) = \frac{1}{r_{w}}\Big(f(p_{w1}) - f(p_{w2})\Big).
\end{equation}
\end{subequations}
The Laplacian at the points $p_0$ and $p_2$ is given by formula 
\eqref{LaplacianPointOldOne} for $m\ge 1$, while for $p=p_1$ is given by
\[
H_mf(p)= \frac{1}{r_w}\Big( f(p_{w2} - f(p_{w1}) + 
h\big(f(p_{w0}) - f(p_{w1})\big) \Big),
\]
$w=(1,1,\ldots,1)\in W_m$ is the word with $m$ ones. Note that $r_w = r_1^m$.

Observe that we have a family of harmonic structures for $K$, parameterized by
$h$. For each $h>1$, the Hausdorff dimension with respect to the effective
resistance metric \cite{Kigami} is the unique $d$ such that
\[
\Big(\frac{1}{h}\Big)^d + \Big(1 - \frac{1}{h^2}\Big)^d = 1. 
\]
We note that, since $F_1$ and $F_2$ are affine linear contractions with
constants $|\alpha|$ and $1 - |\alpha|^2$, respectively, $d$ coincides with the
Hausdorff dimension with respect to the Euclidean metric if $h = 1/|\alpha|$, by
Hutchinson theorem \cite{Hutchinson}.

\section{Harmonic functions}\label{harmonicsection}

A function $u$ on $V_*$ is harmonic if $H_mu(p)=0$ for any $p\in V_m\setminus
V_0$. In this section we describe an algorithm to construct harmonic functions 
on $V_*$ from any boundary values on $V_0$.

We say that $T_w\subset V_m$ is a \emph{minimal cell in $V_m$}, if $T_w$ a set
of three vertices of the form $p_{w0}$, $p_{w1}$ and $p_{w2}$, with $w\in W_m$.
As points in $V_{m+1}$, we have that
\[
T_w = \{p_{w12}, p_{w11}, p_{w22}\},
\]
and the ``new points'' in $V_{m+1}$, contained in $K_w$, are then $p_{w10} =
p_{w21}$ and $p_{w20}$. In other words, $K_w\cap V_{m+1}$ is the union of the
minimal cells in $V_{m+1}$
\[
T_{w1} = \{p_{w10}, p_{w11}, p_{w12}\} \qqand
T_{w2} = \{p_{w20}, p_{w21}, p_{w22}\}
\]

Given a harmonic function on $V_m$, we want to extend to a harmonic function on
$V_{m+1}$. The extension from $T_w$ to $T_{w1}\cup T_{w2}$ is given by the
following algorithm.

\begin{algorithm}\label{alg-harm}
Let $u$ be a harmonic function on $V_m$. If, for each $w\in W_m$, $T_w$ is a
minimal cell in $V_m$, and we extend $u$ to $T_{w1}\cup T_{w2}$ with
\begin{subequations}\label{alg-eqn}
\begin{eqnarray}
u(p_{w10}) &=& \Big(1 - \frac{1}{h^2}\Big) u(p_{w1}) + \frac{1}{h^2}
u(p_{w2}),\label{alg-eqn10}\\
u(p_{w20}) &=&u(p_{w10}),\label{alg-eqn20}
\end{eqnarray}
\end{subequations}
then $u$ is harmonic in $V_{m+1}$.
\end{algorithm}

\begin{proof}
With respect to the basis 
$\{\chi_{p_{12}},\chi_{p_{11}},\chi_{p_{22}},\chi_{p_{10}},\chi_{p_{20}}\}$, 
the matrix of $H_1$ is given by
\[
H_1 =
\begin{pmatrix}
-\dfrac{1}{r_1} & \dfrac{1}{r_1} & 0 & 0 & 0\\
\dfrac{1}{r_1} & - \dfrac{1+h}{r_1} & 0 & \dfrac{h}{r_1} & 0\\
0 & 0 & - \dfrac{1}{r_2} & \dfrac{1}{r_2} & 0\\
0 & \dfrac{h}{r_1} & \dfrac{1}{r_2} & -\dfrac{h}{r_1} - \dfrac{1+h}{r_2} & 
\dfrac{h}{r_2}\\
0 & 0 & 0 & \dfrac{h}{r_2} & - \dfrac{h}{r_2}
\end{pmatrix}.
\]
Writing the matrix as
\[
\begin{pmatrix}
T & J^t \\
J & X
\end{pmatrix},
\]
where $T$ takes functions on $V_0$ to functions on $V_0$, $J$ functions on 
$V_0$ 
to functions on $V_1$, and $X$ functions on $V_1$ to functions on $V_1$, 
Theorem 
2.1.6 in \cite{Kigami} implies that, if $u$ is harmonic, then
\[
u|_{V_1\setminus V_0} = -X^{-1}J(u|_{V_0}).
\]
Multiplying the matrices, and by the compatibility of the sequence of $H_m$ 
\cite{Kigami}, we obtain the result.
\end{proof}

Algorithm \ref{alg-harm} allows us to construct harmonic functions on the Hata
set with arbitrary precision, with any particular value of the parameter $h$.
Figure \ref{graphs-harmonic} shows examples of harmonic functions, with 
distinct 
boundary values and distinct values of $h$.
\begin{figure}[ht]
\includegraphics[width=4in]{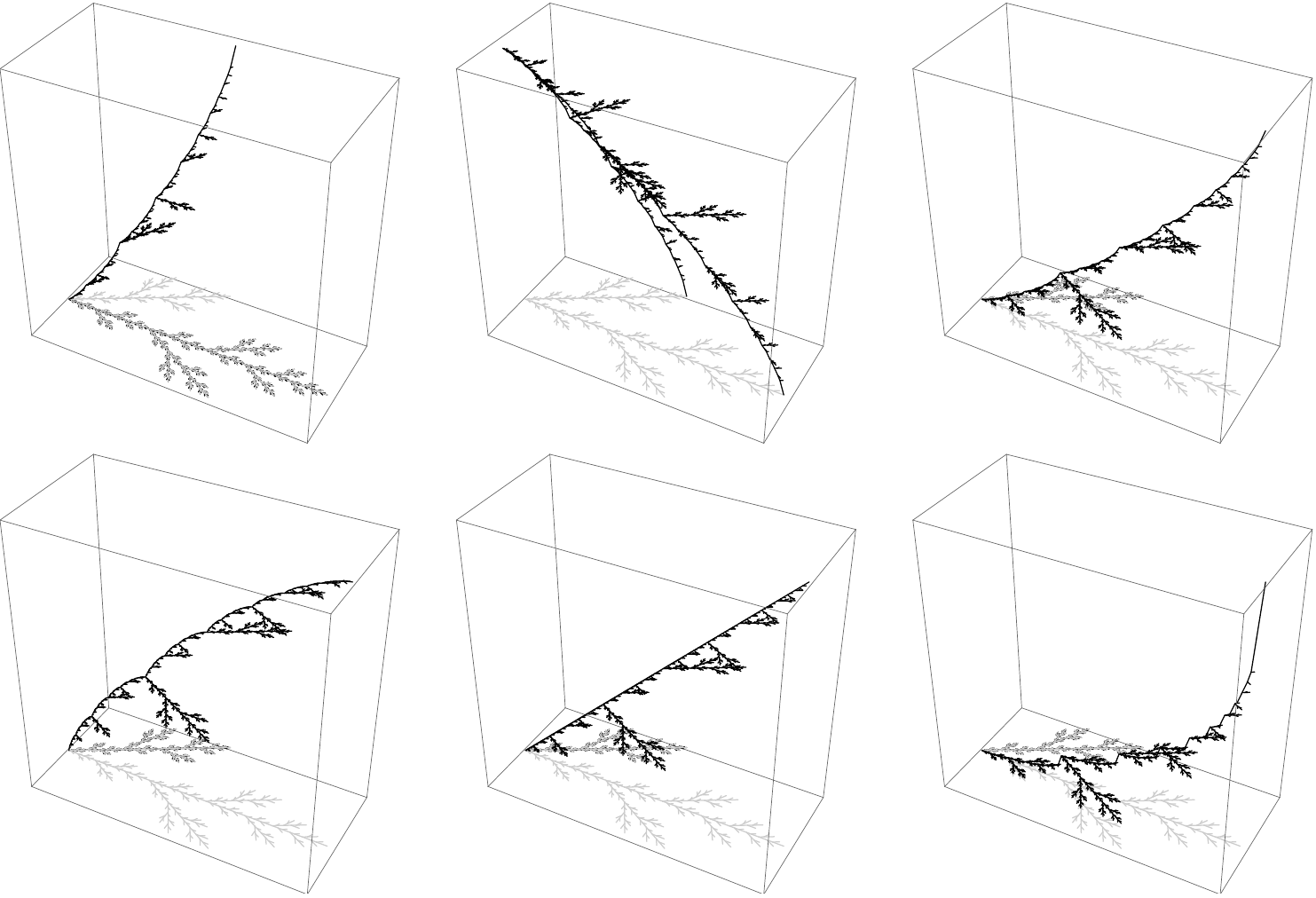}
\caption{Harmonic functions on the Hata set. The three on the top correspond to 
$\alpha=1/2 + \sqrt{3}i/6$ and $h=2$, with boundary values $\chi_c$, $\chi_0$ 
and $\chi_1$, respectively. On the bottom, harmonic funcions with boundary 
values $\chi_1$, with values of $h$ equal to $3/2$, $\sqrt 3$ and $3$, 
respectively.}
\label{graphs-harmonic}
\end{figure}

We observe that, since $K\setminus V_0$ is disconnected, we have harmonic 
functions supported in each one of the connected components $L$ and $M$ of 
$K\setminus V_0$, as we see in Figure \ref{graphs-harmonic} for harmonic 
functions with boundary values $\chi_1$ and $\chi_\alpha$, respectively.

Moreover, from equation \eqref{alg-eqn10}, the value of a harmonic function on 
each point $p$ in the line segment from 0 to 1 only depends on its values on 
the 
adyacent points to $p$ in the same line, from a previous iteration. Thus, one 
can easily construct restrictions of such harmonic
functions to the line segment, following the work of \cite{DSV99}. Figure 
\ref{restr-harmonic} shows examples of such restrictions with boundary values 
$\chi_1$.
\begin{figure}[ht]
\includegraphics[width=4in]{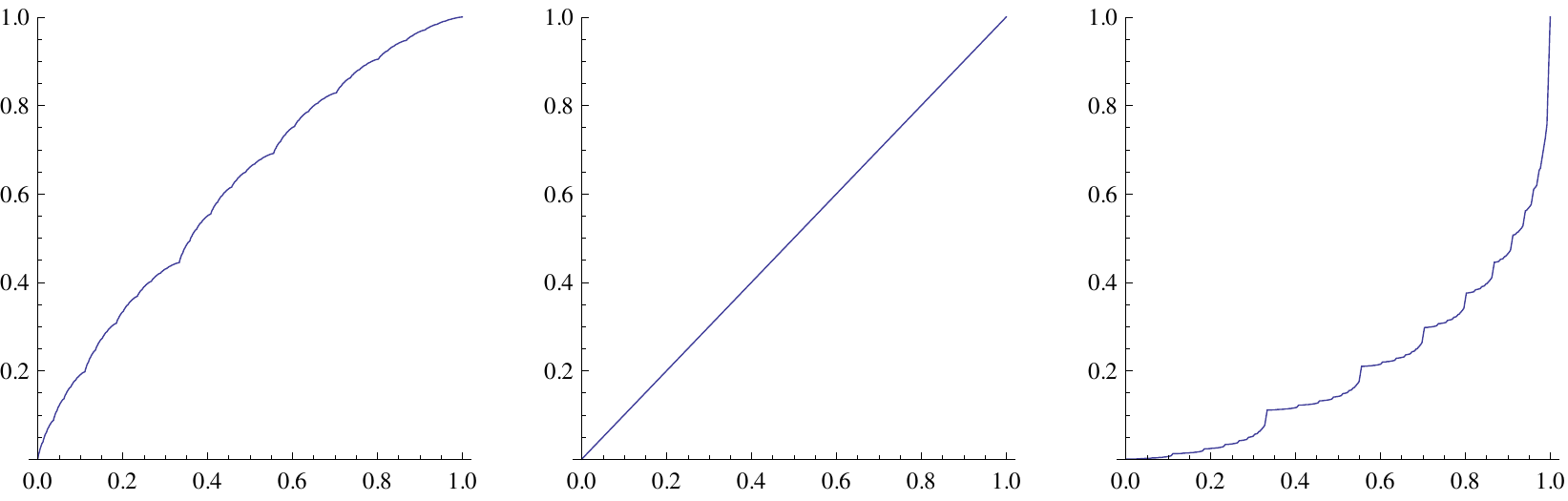}
\caption{Restriction to $[0,1]$ of harmonic functions on the Hata set, with 
$\alpha=1/2 + \sqrt{3}i/6$, boundary values $\chi_1$ and values of $h$ equal to 
$3/2$, $\sqrt 3$ and $3$, respectively. Observe that, in the case 
$h=1/|\alpha|$, such restriction is a line. Otherwise, is a singular function.}
\label{restr-harmonic}
\end{figure}

We note, as readily verified from equation \eqref{alg-eqn10}, that the
restriction of a harmonic function to this line segment is a line if $h =
1/|\alpha|$, because the ``middle'' point in each iteration corresponds to the
convex combination of its adyacent points in the line with weights $|\alpha|^2$
and $1 - |\alpha|^2$. However, if $h \not= 1/|\alpha|$, this restriction is a
singular function, \ie monotonic and with derivative 0 almost everywhere
\cite{YHK97}.

\begin{theorem}\label{harm-sing}
Assume $h \not= 1/|\alpha|$ and let $u$ be a harmonic function on the Hata set
with boundary values $u|_{V_0} = \chi_1$, and let $f$ be its
restriction to the interval $[0,1]$. Then
\begin{enumerate} 
\item $f$ is increasing; and 
\item $f$ is differentiable on $[0,1]\setminus V_*$, with
\[
f'(x) = 0
\]
for every $x\in[0,1]\setminus V_*$.
\end{enumerate}
\end{theorem}

\begin{proof}
Since $|\alpha|^2 = p_{10} = F_1(F_1(1))$,
\[
F_1(F_1(x)) = |\alpha|^2 x \qqand F_2(x) = (1 - |\alpha|^2)x + |\alpha|^2,
\]
we see that equation \eqref{alg-eqn10} implies that $f$ satisfies the system of 
equations
\begin{equation*}
\begin{cases}
f(x) = \dfrac{1}{h^2} f\bigg(\dfrac{1}{|\alpha|^2}x\bigg) & 0 \le x \le 
|\alpha|^2\vspace*{.05in}\\
f(x) = \bigg( 1 - \dfrac{1}{h^2}\bigg) f\bigg(\dfrac{1}{|\alpha|^2}x - 1\bigg) 
+ \dfrac{1}{h^2} & |\alpha|^2 \le x \le 1.
\end{cases}
\end{equation*}
Hence $f$ is essentially the same as Lebesgue's singular function, and the 
result of the theorem follows as in \cite[Section 3.4]{YHK97}.
\end{proof}

We note that, if $h = 1/|\alpha|$, then we have $f(x) = x$. Recall that $d$ 
coincides with the Euclidean Hausdorff dimension precisely when 
$h = 1/|\alpha|$.
\section{Laplacian}\label{Laplaciansection}

In this section we calculate the Laplacian $\Delta$ on the Hata set $K$ with
respect to the self-similar measure $\mu$ with weights
\[
\mu_1 = \mu(F_1(K)) = r_1^d = \Big(\frac{1}{h}\Big)^d \qand
\mu_2 = \mu(F_2(K)) = r_2^d = \Big( 1 - \frac{1}{h^2} \Big)^d.
\]
This measure is comparable to the Hausdorff measure with respect to the 
effective resistance metric \cite{Kigami94}. Moreover, $\mu$ is homogenous with 
respect to this metric \cite{Saenz12}.

As in \cite[Section 3.7]{Kigami}, the domain $\D$ of the operator $\Delta$ is
given by the set of continuous functions $u$ on $K$ such that there exists a
continuous function $f$ with
\begin{equation}\label{Laplacian}
\lim_{m\to\infty} \max_{p\in V_m\setminus V_0}
\big|\frac{1}{\mu_p^m}H_mu(p) - f(p) \big| = 0,
\end{equation}
where $\mu_p^m=\int_K \psi_p^m d\mu$, the integral of the $m$-harmonic spline
$\psi_p^m$ that satisfies 
\[
\psi_p^m(p) = 1 \qqand \psi_p^m(q) = 0 \text{ for }q\in V_m, q\not=p.
\]
If $u\in\D$ and $f$ is as in \eqref{Laplacian}, then we write $\Delta u = f$.
We can then approximate explicitly $\Delta u$ once we calculate the
normalizing numbers $\mu_p^m$.

In order to calculate these numbers we observe that, by the self-similarity of
$\mu$, 
\begin{equation*}
\begin{split}
\mu_p^m &= \int_K \psi_p^m d\mu
= \sum_{\substack{w\in W_m\\p\in K_w}} \int_{K_w}\psi_p^m d\mu\\
&= \sum_{\substack{w\in W_m\\p\in K_w}} \mu_w \int_K \psi_{F_w^{-1}(p)}^0 d\mu
= \sum_{\substack{w\in W_m\\p\in K_w}} \mu_w \mu_{F_w^{-1}(p)}^0,
\end{split}
\end{equation*}
where $\mu_w = \mu_{w_1}\cdots\mu_{w_m}$. Thus it is sufficient to calculate
the three numbers $\mu_\alpha^0, \mu_0^0$ and $\mu_1^0$ corresponding to
$\alpha$, $0$ and $1$, the points in $V_0$.

Again, using the self-similarity of $\mu$, we observe that, by Algorithm 
\ref{alg-harm}
\[
\int_K \psi_\alpha^0 d\mu = \mu_1 \int_K \psi_1^0 d\mu,
\]
\begin{multline*}
\int_K \psi_0^0 d\mu = \mu_1 \int_K 
\bigg(\Big(1 - \frac{1}{h^2}\Big)\psi_\alpha^0 + \psi_0^0 \bigg) d\mu \\+ 
\mu_2 \int_K 
\bigg( \Big(1 - \frac{1}{h^2}\Big)\psi_\alpha^0 + \Big(1 -
\frac{1}{h^2}\Big)\psi_0^0 \bigg) d\mu,
\end{multline*}
and
\[
\int_K \psi_1^0 d\mu = \mu_1 \int_K \frac{1}{h^2} \psi_\alpha^0 d\mu + 
\mu_2 \int_K \Big( \frac{1}{h^2} \psi_\alpha^0 + \frac{1}{h^2} \psi_0^0 +
\psi_1^0
\Big) d\mu,
\]
Thus, $\mu_c^0$, $\mu_0^0$ and $\mu_1^0$ satisfy the system of equations
\[
\begin{split}
\mu_\alpha^0 &= \mu_1 \mu_1^0\\
\mu_0^0 &= \Big(1 - \frac{1}{h^2}\Big)\mu_\alpha^0 + 
\mu_2 \Big(1 - \frac{1}{h^2}\Big)\mu_0^0\\
\mu_1^0 &= \frac{1}{h^2}\mu_\alpha^0 + \frac{\mu_2}{h^2}\mu_0^0 + \mu_2 \mu_1^0,
\end{split}
\]
where we have already use the fact $\mu_1 + \mu_2 = 1$. Moreover, as the sum
$\psi_\alpha^0 + \psi_0^0 + \psi_1^0$ is the constant function 1, we also have
\[
\mu_\alpha^0 + \mu_0^0 + \mu_1^0 = 1.
\]
Solving this system we obtain 
\[
\mu_\alpha^0 = \frac{\mu_1\mu_2}{\mu_1\mu_2 + (h^2-1)\mu_1 + \mu_2},
\quad
\mu_0^0 = \frac{(h^2-1)\mu_1}{\mu_1\mu_2 + (h^2-1)\mu_1 + \mu_2},
\]
and
\[
\mu_1^0 = \frac{\mu_2}{\mu_1\mu_2 + (h^2-1)\mu_1 + \mu_2}.
\]

\section{Dirichlet Spectrum}\label{Dirichletsection}

We now proceed to study the Dirichlet spectrum of $\Delta$. As there is no 
decimation on the Hata set, we have to use the finite element method in order 
to approximate the eigenvalues and eigenfunctions of $\Delta$. We present a 
summary of the observations obtained numerically by solving the system of 
equations
\[
\begin{cases}
-\Delta_m u(x) = \lambda u(x) & x\in V_m\setminus V_0 \\
u(p) = 0 & p\in V_0,
\end{cases}
\]
where $\Delta_m u(x) = \frac{1}{\mu_p^m}H_mu(x)$.

Recall that $L$ is the closure of the connected component of $K\setminus V_0$ 
that contains the interval $(0,1)$, and $M$ the closure of the component that 
contains the open segment from $0$ to $\alpha$. Thus, linearly independent 
Dirichlet eigenfunctions are supported either in $L$ or $M$.

We observe numerically that the Dirichlet ground state is supported in $L$, as 
observed in Figure \ref{ground-state} (for $h=2$), corresponding to 
$\lambda_1 \approx 9.888$.
\begin{figure}[ht]
\includegraphics[width=2in]{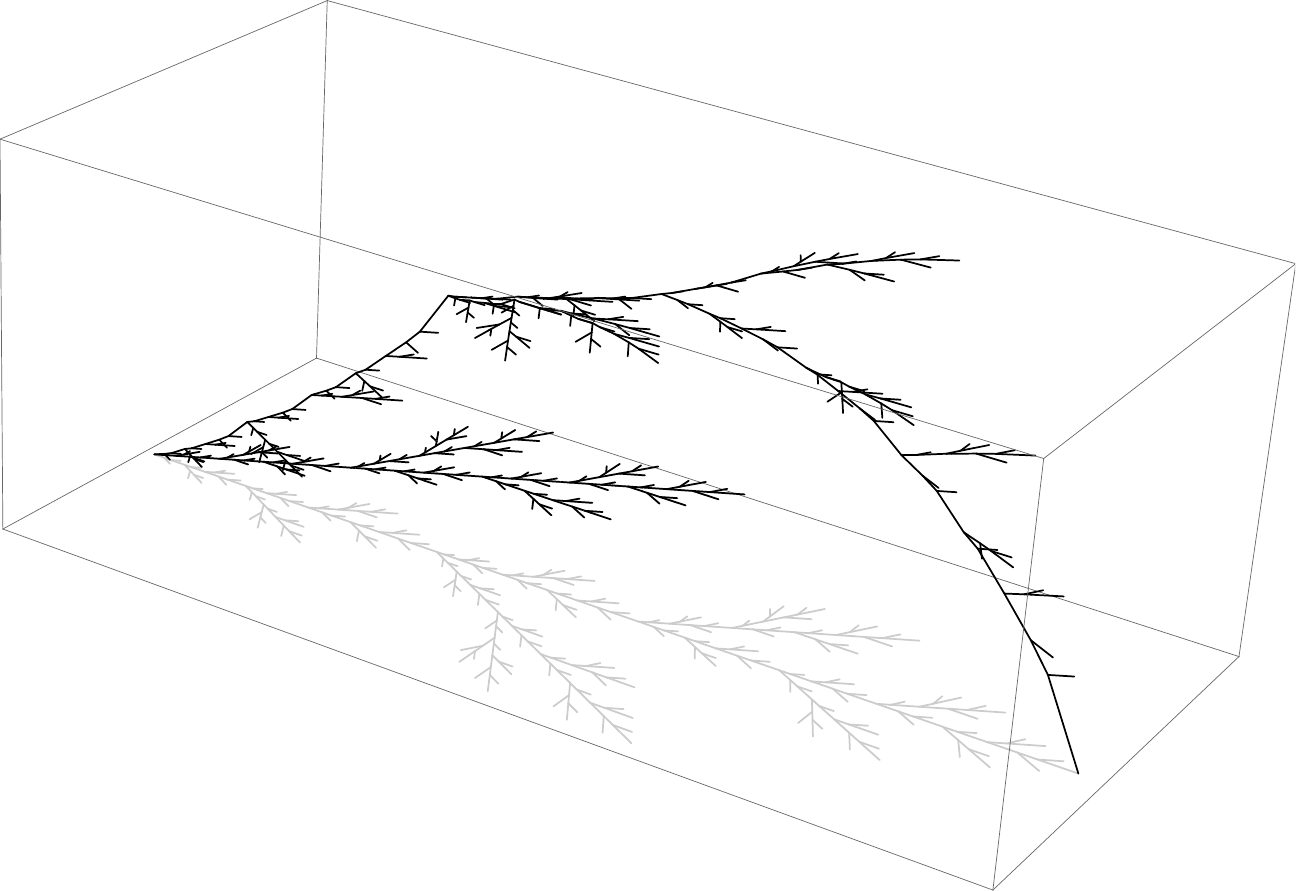}
\includegraphics[width=2in]{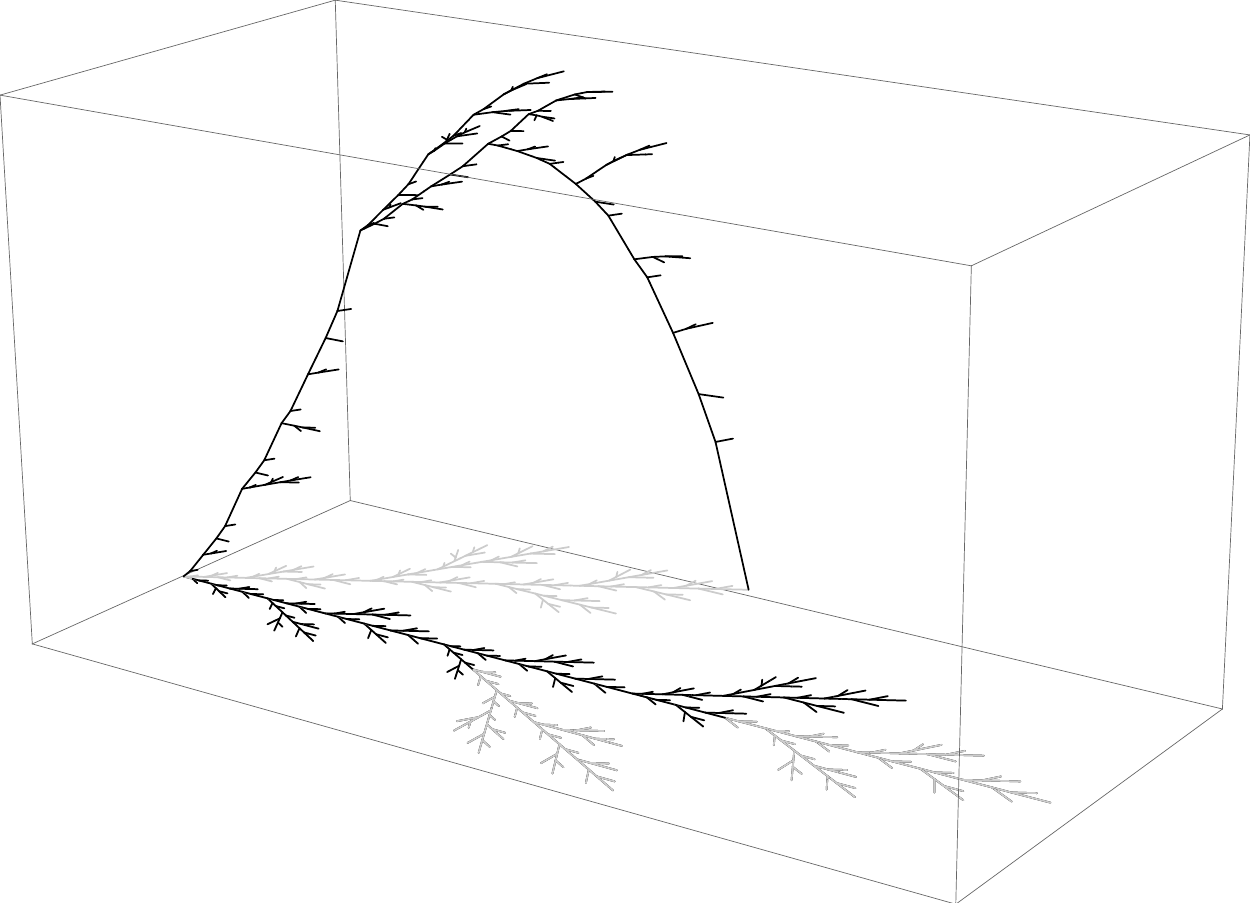}
\caption{Dirichlet ground state $\phi$ for $h=2$, $\lambda_1 \approx 
9.888$, and its derived eigenfunction $\tilde\phi = \chi_{K_1}\cdot\phi\circ 
F_1^{-1}$, $\lambda_3 \approx 56.21$, approximated to the eigth iteration, 
with $\alpha=1/2 + \sqrt{3}i/6$. Note that $\phi$ is supported in $L$ and 
$\tilde\phi$ is supported in $M$.}
\label{ground-state}
\end{figure}

For each eigenfunction $\phi$ supported in $L$, there is a corresponding 
eigenfunction supported in $M$.

\begin{proposition}\label{derived-eigenvalues}
Let $\phi$ be a Dirichlet eigenfunction of $\Delta$, supported in $L$, with 
respect to the eigenvalue $\lambda$. Then $\chi_{K_1}\cdot\phi\circ F_1^{-1}$ 
is an eigenfunction supported in $M$ with respect to the eigenvalue 
$\lambda/(r_1\mu_1)$.
\end{proposition}

\begin{proof}
Let $\tilde\phi = \chi_{K_1}\cdot\phi\circ F_1^{-1}$. If $x\in L$, then $x\in 
K_2$ or $x\in L\cap K_1$. In the first case, clearly $\tilde\phi(x)=0$ because 
$\chi_{K_1}(x)=0$, unless $x=|\alpha|^2$, the critical point. But in that case 
$x\in L\cap K_1$, and $F_1^{-1}(x)\in M$, so $\phi\circ F_1^{-1}(x)=0$ since 
$\phi$ is supported in $L$. Therefore $\tilde\phi$ is supported in $M$.

Now, for $u\in\F_0$, by the self-similarity of the Dirichlet form $\dir$,
\[
\begin{split}
\dir(\tilde\phi, u) &= \frac{1}{r_1}\dir(\tilde\phi\circ F_1, u\circ F_1) + 
\frac{1}{r_2}\dir(\tilde\phi\circ F_2, u\circ F_1)\\
&= \frac{1}{r_1}\dir(\phi, u\circ F_1) 
= \frac{\lambda}{r_1}\int_K \phi \, u\circ F_1 d\mu,
\end{split}
\]
where we have used the fact that $\tilde\phi$ is supported in $K_1$ and $\phi$ 
is a Dirichlet eigenfunction with respect to $\lambda$. One the other hand, by 
the self-similarity of the measure $\mu$,
\[
\int_K \tilde\phi\, u d\mu = \mu_1 \int_K \tilde\phi\circ F_1 \, u\circ F_1 d\mu
+ \mu_2 \int_K \tilde\phi\circ F_2 \, u\circ F_2 d\mu =
\mu_1 \int_K \phi \, u\circ F_1 d\mu,
\]
so we obtain
\[
\dir(\tilde\phi, u) = \frac{\lambda}{r_1\mu_1} \int_K \tilde\phi\, u d\mu,
\]
and thus we conclude $\Delta\tilde\phi = - \dfrac{\lambda}{r_1\mu_1}\tilde\phi$.
\end{proof}

Figure \ref{ground-state} (right) shows the eigenfunction corresponding to the 
eigenvalue $\lambda_3 = \lambda_1/r_1\mu_1 \approx 56.21$, where $\lambda_1$ 
is the first Dirichlet eigenvalue ($h=2$). Proposition 
\ref{derived-eigenvalues} lets us classify the Dirichlet eigenvalues (and 
eigenfunctions) in two classes, which we will call ``primary'' and ``derived''. 
Table \ref{eigentable} shows the approximations to the first 20 Dirichlet 
eigenvalues, for $h=3/2,\sqrt 3$ and $3$. We note that the derived eigenvalues 
appear in different positions in the sequence $\lambda_k$, depending on $h$, 
and 
they seem to be more sparse as $h$ increases.
\begin{table}[ht]
\caption{Approximations $\lambda_k^{10}$ to the first 20 Dirichlet eigenvalues 
for values of $h$ equal to $3/2$, $\sqrt 3$ and $3$, indicating whether they 
correspond to a primary or a derived eigenfunction.}
\label{eigentable}
\begin{tabular}{|c|c|c|c|c|c|}
\hline
\multicolumn{2}{|c|}{$h=3/2$} & \multicolumn{2}{|c|}{$h=\sqrt 3$} 
&\multicolumn{2}{|c|}{$h=3$} \\
\hline
$\lambda_k^{10}$ & Type & $\lambda_k^{10}$ & Type &$\lambda_k^{10}$ & Type \\ 
\hline
2.12748 & Primary & 10.012 & Primary & 38.7802 & Primary \\
5.80965 & Derived & 31.037 & Primary & 139.978 & Primary \\
8.3776 & Primary & 38.7455 & Derived & 255.362 & Primary \\
13.7502 & Primary & 83.3496 & Primary & 336.428 & Primary \\
22.8762 & Derived & 106.366 & Primary & 435.129 & Primary \\
33.3334 & Primary & 120.11 & Derived & 566.447 & Primary \\
34.0196 & Primary & 193.982 & Primary & 741.34 & Primary \\
37.5447 & Derived & 226.027 & Primary & 972.052 & Primary \\
53.6119 & Primary & 244.389 & Primary & 1067.3 & Derived \\
59.5265 & Primary & 322.541 & Derived & 1266.44 & Primary \\
91.007 & Derived & 401.184 & Primary & 1623.74 & Primary \\
92.8821 & Derived & 411.618 & Derived & 1814.55 & Primary \\
98.8091 & Primary & 503.566 & Primary & 2248.34 & Primary \\
109.503 & Primary & 580.894 & Primary & 2574.77 & Primary \\
133.249 & Primary & 613.579 & Primary & 2909.76 & Primary \\
136.474 & Primary & 654.576 & Primary & 3013.28 & Primary \\
146.338 & Derived & 750.644 & Derived & 3299.2 & Primary \\
162.469 & Derived & 783.032 & Primary & 3812.6 & Primary \\
195.591 & Primary & 874.566 & Derived & 4001.01 & Derived \\
213.997 & Primary & 945.709 & Primary & 4147.5 & Primary \\
\hline
\end{tabular}
\end{table}
Figure \ref{eigenfunctionsm10} shows the first four Dirichlet eigenfunctions 
for 
those values of $h$, where we can observe the appearance of the derived 
eigenfunctions corresponding to $\lambda_2^{10}$ ($h=3/2$) and $\lambda_3^{10}$ 
($h=\sqrt 3$).
\begin{figure}[ht]
\includegraphics[width=4.5in]{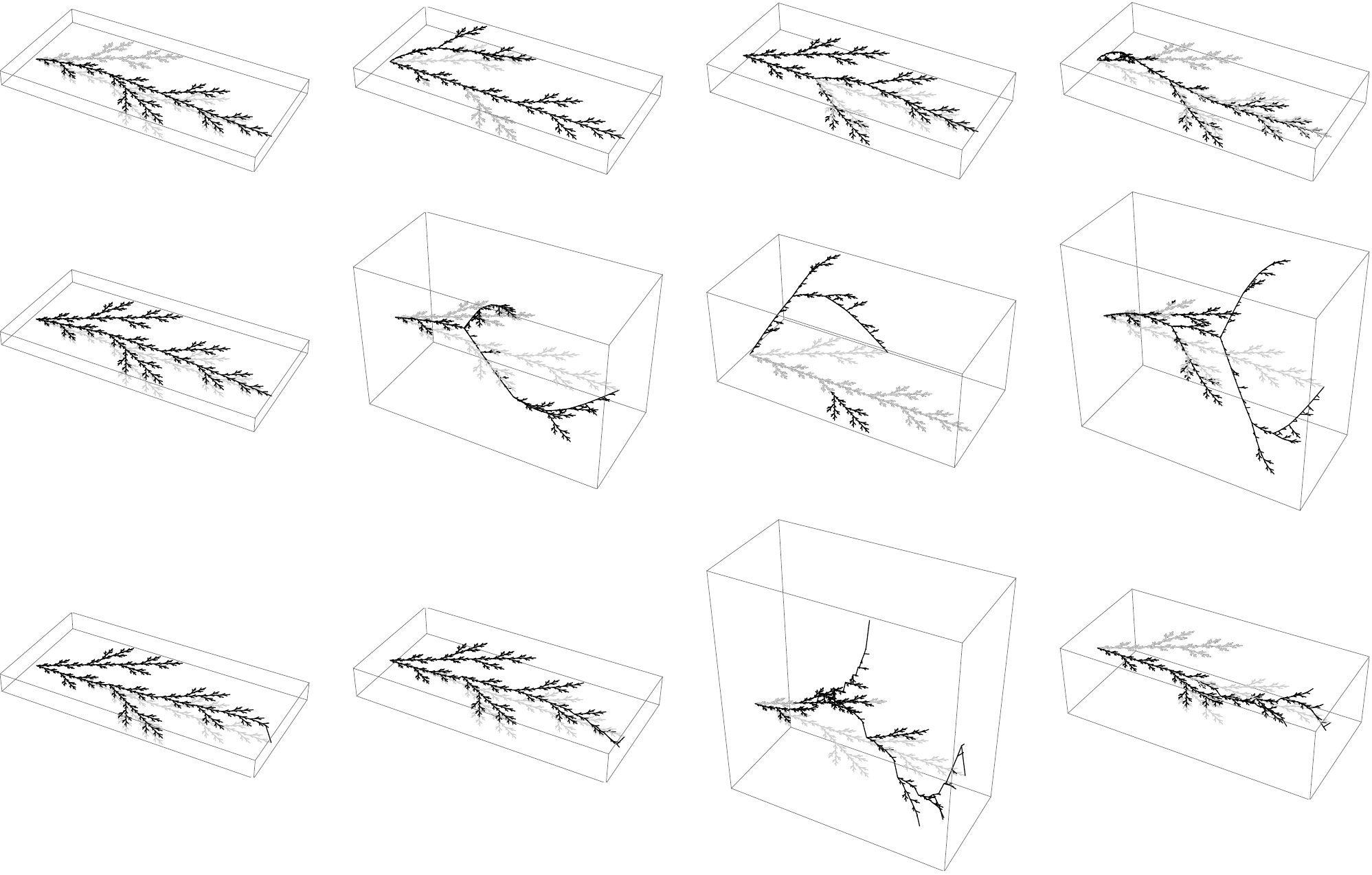}
\caption{The first four Dirichlet eigenfunctions for the values $h=3/2,\sqrt 3$ 
and $3$, respectively, approximated to the 10th iteration, with $\alpha=1/2 + 
\sqrt{3}i/6$.}
\label{eigenfunctionsm10}
\end{figure}

More interestingly, we show in Figure \ref{eigenfunctionsm10-rest} the 
\begin{figure}[ht]
\includegraphics[width=4.5in]{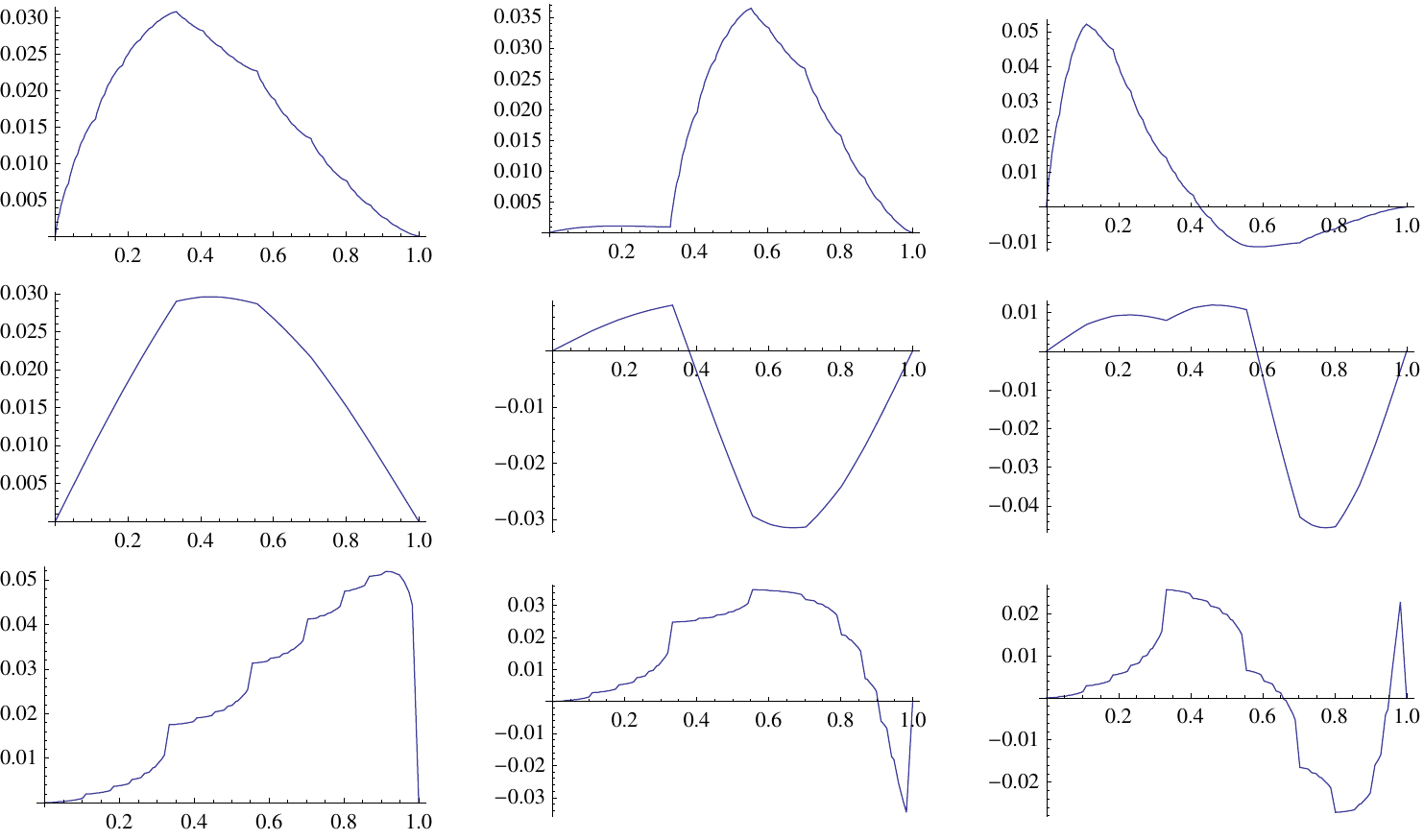}
\caption{The restrictions to $[0,1]$ for the first three primary Dirichlet 
eigenfunctions for the values $h=3/2,\sqrt 3$ and $3$, respectively, 
approximated to the 10th iteration, with $\alpha=1/2 + \sqrt{3}i/6$.}
\label{eigenfunctionsm10-rest}
\end{figure}
restrictions of these eigenfunctions to the interval $[0,1]$ (only the primary 
ones, as the derived are zero in $[0,1]$). One can ask 
whether these functions have singularity properties as in the case of the 
harmonic functions. 

Recall that, in the case of the harmonic functions, if $q\in V_{m+1}\cap[0,1]$ 
is the middle point between the points $x,y\in V_m\cap[0,1]$, then the 
harmonic function $u$ satisfies
\begin{equation}\label{proportions-eq}
u(q) = (1-\theta) u(x) + \theta u(y),
\end{equation}
where $\theta = 1/h^2$. In Figure \ref{proportions}, we show the
\begin{figure}[ht]
\includegraphics[width=4.5in]{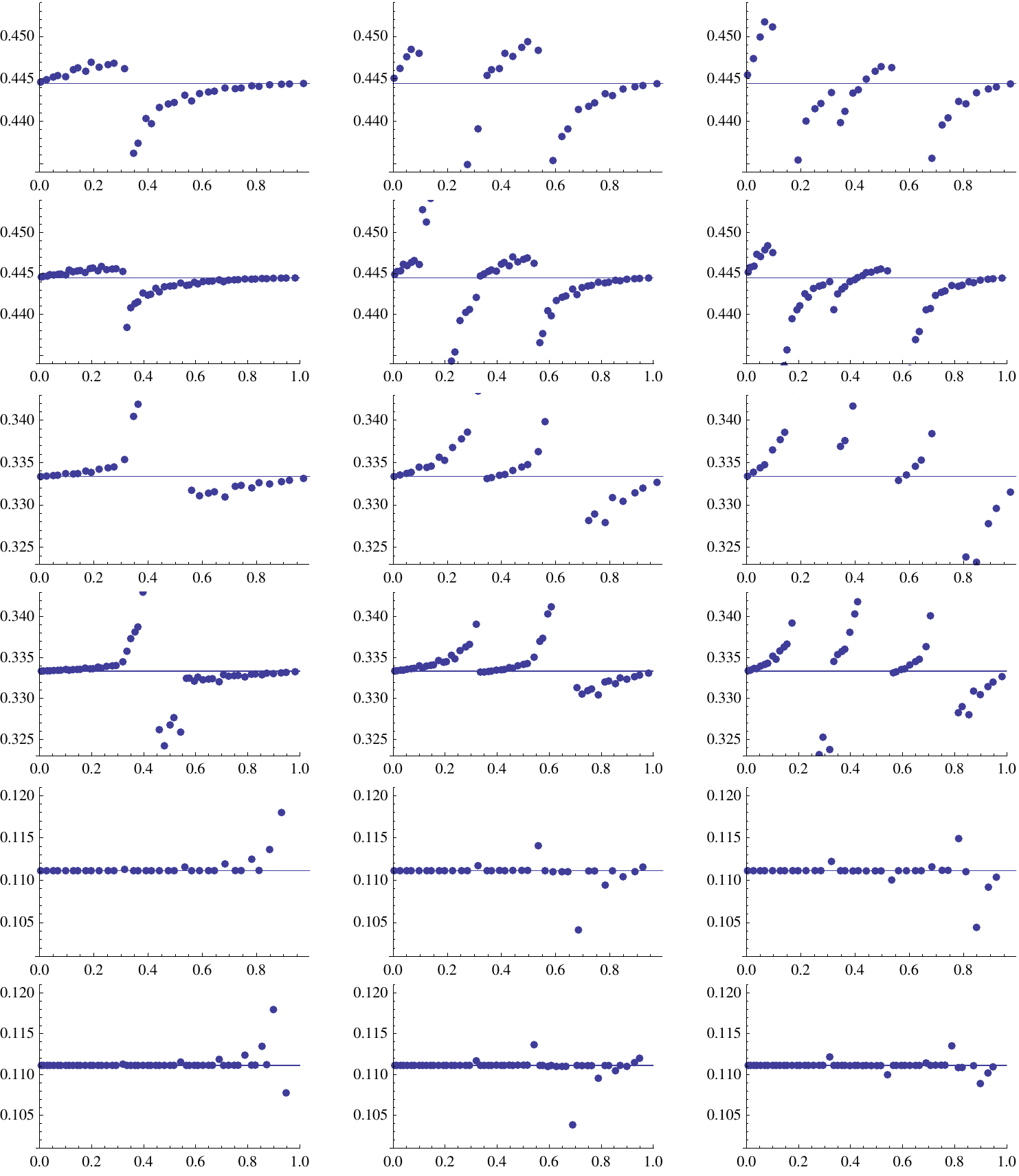}
\caption{The values of $\theta$ that solve equation \eqref{proportions-eq} for 
each middle point $q\in V_9\cap[0,1]$ (top of each pair of graphs) and $q\in 
V_{10}\cap[0,1]$ (bottom), for the approximations at $m=9$ and $m=10$ to the 
restrictions to $[0,1]$ for the first three primary Dirichlet eigenfunctions 
for 
the values $h=3/2,\sqrt 3$ and $3$, respectively, with $\alpha=1/2 + 
\sqrt{3}i/6$. The line corresponds to $1/h^2$, and the plot range is $1/h^2 \pm 
.01$.}
\label{proportions}
\end{figure}
values of $\theta$ for such middle points $q\in V_9\cap[0,1]$ and $q\in 
V_{10}\cap[0,1]$, with respect to their adjacent points in $V_8$ and $V_9$, 
respectively, for the approximations at $m=9$ and $m=10$ to the restrictions to 
$[0,1]$ for the first three primary Dirichlet eigenfunctions for the values 
$h=3/2,\sqrt 3$ and $3$.

We observe that they are closely equal to $1/h^2$, with better approximations 
as 
$h$ increases. We show the two iterations $m=9$ and $m=10$ in order to find out 
if the same proportions are preserved through two levels. As the $\theta$ are 
preserved through different iterations, we are lead to conjecture that the 
restrictions to $[0,1]$ of the Dirichlet eigenfunctions of the Laplacian with 
respect to $h\not=1/|\alpha|$ are singular functions whenever they are 
monotone, as in the case of the harmonic functions.

\section{Conclusions}\label{Conclusionssection}

We have studied harmonic and Dirichlet eigenfunctions for the family of 
harmonic structures of the Hata set $K$ parametrized by $h>1$. The former 
can be constructed by means of the known algorithms for harmonic functions on 
PCF sets, and one observes that, when restricted to the interval $[0,1]$, the 
longest edge in $K$, one obtains singular functions for all but one value of the 
parameter $h$. In fact, we observed that the only value of $h$ for which these 
restrictions are not singular coincides with the value such that the Hausdorff 
dimensions of $K$ with respect to the Euclidean and effective resistance 
metrics are the same.

As is known, restrictions of harmonic functions to the edges of the 
Sierpinski gasket are also singular. This lead us to ask whether this behaviour 
is typical for harmonic functions on PCF sets. Moreover, since we know that in 
the case of the Hata set $K$ such functions are not singular for a particular 
embedding $K$ in the plane (given $h$, one can choose $\alpha$ such that 
$|\alpha|=1/h$), one can also ask whether, for every PCF set, there exists an 
embedding such that these restrictions are not singular. In particular, is this 
true for the Sierpinski gasket?

The same questions can be asked for the eigenfunctions of a PCF set. We have 
numerical evidence for the case of the Hata set, and one can start by studying 
those PCF sets with decimation. In particular, do the eigenfunctions on the 
Sierpinski gasket have singular restrictions to the edges?






\end{document}